\documentclass[12pt]{amsart}

\usepackage{color}

\usepackage{pb-diagram,pb-xy,bm}
\usepackage{amssymb,latexsym, 
amsmath, amscd, array, graphicx, enumerate}
\usepackage[pdftex,bookmarks=true,bookmarksnumbered=true,colorlinks=false,hidelinks=true]{hyperref}
\usepackage[all]{xy}
\CompileMatrices

\theoremstyle{plain}
\newtheorem{thm}{Theorem}[section]
\newtheorem{theorem}[thm]{Theorem}

\newtheorem{lemma}[thm]{Lemma}
\newtheorem{conjec}[thm]{Conjecture}
\newtheorem{prop}[thm]{Proposition}
\newtheorem{cor}[thm]{Corollary}

\theoremstyle{definition}
\newtheorem{defin}[thm]{Definition}

\newtheorem{rem}[thm]{Remark}
\newtheorem{remark}[thm]{Remark}

\newtheorem{ex}[thm]{Example}

\newcommand\theoref{Theorem~\ref}
\newcommand\propref{Proposition~\ref}
\newcommand\lemref{Lemma~\ref}
\newcommand\corref{Corollary~\ref}
\newcommand\defref{Definition~\ref}

\newcommand\exref{Example~\ref}

\def\secref{Section~\ref}

\def\CF{\mathrm{Conf}}
\def\B{\mathrm{Braid}}
\def\DF{{D}}
\def\P{\mathrm{P}}
\def\Sv{\protect{\mathfrak{genus\,}}}
\def\cat{\protect\operatorname{cat}}
\def\hdim{\protect\operatorname{hdim}}
\def\conn{\protect\operatorname{conn}}
\def\cl{\protect\operatorname{cl}}
\def\TC{\protect\operatorname{TC}}
\def\STC{\protect\operatorname{TC}^{\Sigma}}
\def\eps{\varepsilon}
\def\gf{\varphi}
\def\ga{\alpha}
\def\gb{\beta}
\def\gg{\gamma}

\def\gr{\rho}

\def\Si{\Sigma}
\def\Z{{\mathbb Z}}
\def\R{{\mathbb R}}
\def\cp{{\mathcal P}}
\def\empt{\varnothing}
\def\m{\medskip}
\def\ov{\overline}
\def\ts{\times}
\def\th{^\mathrm{th}}
\def\cupr{\smallsmile}

\begin{document}

\title{Higher topological complexity and its symmetrization}
\author[I. Basabe, J.~Gonz\'alez, Yu.~Rudyak, and D.~Tamaki]{Ibai Basabe, Jes\'us Gonz\'alez, Yuli B. Rudyak, and Dai Tamaki}

\address{Ibai Basabe \newline
Department of Mathematics, University of Florida \newline
358 Little Hall, Gainesville, FL 32611-8105, USA}
\email{iebasabe1@ufl.edu}

\address{Jes\'us Gonz\'alez \newline
Departamento de Matem\'aticas, CINVESTAV-IPN \newline
A.P.~14-740, M\'exico City 07000, M\'exico}
\email{jesus@math.cinvestav.mx}

\address{Yuli B. Rudyak \newline
Department of Mathematics, University of Florida \newline
358 Little Hall, Gainesville, FL 32611-8105, USA}
\email{rudyak@ufl.edu}

\address{Dai Tamaki \newline
Department of Mathematical Sciences, Shinshu University \newline
Matsumoto, 390-8621, Japan}
\email{rivulus@shinshu-u.ac.jp}

\begin{abstract}
We develop the properties of the $n$-th sequential topological complexity $\TC_n$, a homotopy invariant introduced by the third author as an extension of Farber's topological model for studying the complexity of motion planning algorithms in robotics. We exhibit close connections of $\TC_n(X)$ to the Lusternik-Schnirelmann category of cartesian powers of~$X$, to the cup-length of the diagonal embedding $X\hookrightarrow X^n$, and to the ratio between homotopy dimension and connectivity of $X$. We fully compute the numerical value of $\TC_n$ for products of spheres, closed 1-connected symplectic manifolds, and quaternionic projective spaces. Our study includes two symmetrized versions of $\TC_n(X)$.  The first one, unlike Farber-Grant's symmetric topological complexity, turns out to be a homotopy invariant of $X$; the second one is closely tied to the homotopical properties of the configuration space of cardinality-$n$ subsets of $X$. Special attention is given to the case of spheres.
\end{abstract}
\maketitle

\vspace{-3mm}
{\small {\it 2010 Mathematics Subject Classification:}
Primary 55M30. Secondary 55R80, 55R05, 57Q40, 68T40.

}

\tableofcontents

\section{Introduction, main results, and organization}\label{intro}

A {\it motion planning algorithm} (mpa) for an autonomous system (robot) $S$ is a rule assigning, to each pair $(A,B)$ of initial-final positions of $S$, a (continuous) motion from $A$ to $B$~\cite{L,LV}. If $X$ stands for the space of all possible states of $S$, and $P(X)$ is the space of all paths $\gamma:[0,1]\rightarrow X$, then a mpa for $S$ is a (non-necessarily continuous) section for the end-points evaluation map $e\colon P(X)\rightarrow X\times X$ defined as $e(\gamma)=(\gamma(0),\gamma(1))$. 

\m For practical applications one is interested in continuous mpa's. However it is easy to see that the end-points evaluation map $e$ admits a continuous section if and only if the space of states $X$ is contractible. The alternative is to look at the Schwarz genus of the map $e$, which leads to Farber's concept of topological complexity. This gives a way of recognizing mpa's with the least possible order of instability (see~\cite[Section~4]{Finst}). The recognition is done directly from the homotopical properties of the space of states of the robot.

\m{\bf Definition}~(Farber).
Given a path-connected topological space $X$, the {\it topological complexity} of $X$, $\TC(X)$, is the least positive integer $k$ such that the cartesian product $X\times X$ can be covered  by $k$ open subsets $U_1, U_2,\ldots, U_k$ on each of which $e$ admits a continuous section $s_i\colon U_i\to P(X)$. Each pair $(U_i,s_i)$ is called a {\it local motion planner} with domain $U_i$. We set $\TC(X)=\infty$ if no such $k$ exists.

\m
A symmetrized version of topological complexity arises when attention is restricted to local planners for which the motion from $A$ to $B$ is the reverse of the motion from $B$ to $A$,~\cite{FGsym}. A number of properties of topological complexity and symmetric topological complexity can be found in~\cite{F1,F2,F3,FGsym,FG,FY}. The papers~\cite{FTY,GL} identify these concepts in the case of real projective spaces as their immersion and embedding dimensions, respectively.

\m
This paper is concerned with the third author's generalization of the above concepts. In such a view, the motion planning does not only depend on a couple of initial-final states of a robot, but in a sequence of prescribed intermediate stages that the robot should reach through the motion. Such a setting is standard in industrial production processes in which the manufacture of a given good goes through a series of production steps. The corresponding need to identify best possible sequential motion planning algorithms leads to a homotopy invariant $\TC_n(X)$, the {\it $n$-th topological complexity} of $X$, introduced in \cite{R} and reviewed in \secref{s:prel} (where we use normalized notation, i.e.~ in such a way that contractible spaces have $\TC_n=0$.)

\m
In Section~\ref{s:prop} we discuss basic properties of $\TC_n$, including methods for calculating this homotopy invariant. In \theoref{t:clest} we describe optimal bounds for $\TC_n(X)$: lower bounds are given in terms of the cup-length of elements in the kernel of the iterated diagonal, whereas connectivity and homotopy dimension of~$X$ lead to upper bounds. The subadditivity of $\TC_n$ is settled in~\propref{p:product}. As an application, we obtain the full determination of the numerical value of $\TC_n(X)$ when $X$ is either a product of spheres (\corref{c:spheres}), a closed simply connected symplectic manifold (\corref{c:symplec}), or a quaternionic projective space (\corref{c:projective}).

\m
Many of our results generalize existing properties for Farber's $\TC$. For instance, in \corref{c:bounds} we show the following close connection between higher topological complexity and the Luster\-nik-Schnirelmann category of cartesian powers of spaces:

\m
{\bf Theorem. } {\it For a path-connected space $X$,}
$\cat(X^{n-1})\leq\TC_n(X)\leq\cat(X^n)$.

\m
\theoref{t:topologicalgroup} below gives $\TC_n(G)=\cat(G^{n-1})$ for a path-connected topological group $G$, which extends the $n=2$ property proved by Farber in \cite[Lemma~8.2]{Finst}. Lupton and Scherer have recently proved that  this property extends to not-necessarily homotopy-associative Hopf spaces (see~\cite{LSh}). 

\m
\secref{s:sym} deals with symmetric versions of higher topological complexity. We begin by introducing $\STC(X)$, a variation of the symmetric topological complexity $\TC^S(X)$ introduced in \cite{FGsym}. We prove that the numerical values of the two invariants differ at most by a unit (\propref{p:compare}). Such a fact should be prised by noticing that, although Farber and Grant observe that $\TC^S(X)$ is not a homotopy invariant, $\STC(X)$ depends only on the homotopy type of $X$. It should be noted that the homotopy invariance also fails in general for the {\it monoidal topological complexity} introduced by Iwase and  Sakai (see~\cite[Definition 1.3 and Remark~1.4]{IS}), where the stasis property is imposed on the motion planning problem, instead of the symmetry condition we impose on $\STC$. We construct the corresponding higher analogues $\TC^S_n$ and $\STC_n$, and prove the homotopy invariance of the latter (\propref{p:hinv}). 

\m
The calculation of $\TC_n^S$ can turn out to be an extremely difficult task, mainly due to what seems to be a limited current knowledge of precise homotopy information about braid spaces (even braid manifolds, for that matter). In \secref{s:bounding} we exhibit evidence leading to conjecture that
\begin{equation}\label{conejo}
\TC_n^S(S^k)\leq \left[(n+2)(k-1)+4\rule{0mm}{3.4mm}\right](n-1)/2k
\end{equation}
holds for integers $k\geq1$ and $n\geq2$. In particular, we observe in Corollary~\ref{laevi} that the equality $\TC^S_n(S^k)=2(n-1)$ holds provided $n=2$ or $k=1$.

\m
{\bf Acknowledgments.}
The first author wishes to thank Michael Farber, Jes\'us Gonz\'alez, Dirk Sch\"utz and the Mathematisches Forschungsinstitut Oberwolfach for organizing a wonderful Arbeitsgemeinschaft mit aktuellem Thema in Topological Robotics. The third author is grateful for the support during a visit at the Max-Planck Institute for Mathematics in Bonn, Germany. The fourth author would like to thank the Centro di Ricerca Matematica Ennio De Giorgi, Scuola Normale Superiore di Pisa, for supporting his participation in the research program ``Configuration Spaces:~Geometry, Combinatorics and Topology'', during which a part of his work on this paper was done. The second, third and fourth authors were partially supported, respectively, by Conacyt Research Grant 102783, Simons Foundation Grant 209424, and Grants-in-Aid for Scientific Research, Ministry of Education, Culture, Sports, Science
and Technology, Japan: 23540082. The authors wish to express their most sincere gratitude to Peter Landweber for valuable suggestions on earlier versions of this paper, and for pointing out an important extension of the authors' original evidence for the conjectural assertion in~(\ref{c:TCnSpherek}).

\section{Preliminaries on notation}\label{s:prel}
We use the normalized version of Schwarz's concept of the genus of a map~\cite{Sv}.

\begin{defin}\label{d:sgf}
The {\it Schwarz genus} (also known as {\it sectional category}) of a map $p:E\rightarrow B$ is the least number $k$ such that there is an open covering $U_0, U_1,\ldots,U_k$ of $B$ for which the restriction of $p$ to each $U_i$ ($i=0, 1,\ldots,k$) admits a homotopy section, i.e.~a Ê(continuous) Êmap Ê$s_i: ÊU_i Ê\to ÊE$ such Êthat Ê$ps_i$ Êis Êhomotopic Êto Êthe inclusion $U_i\hookrightarrow B$. We agree to set $\Sv(f)=-1$ for $f: X \to Y$ with $X=\varnothing=Y$.
\end{defin}

The following result, proved in~\cite[Proposition~22 on page 84]{Sv} (see also the comments in Section~1 on page 54 of~\cite{Sv}), will be used in the proof of \propref{p:product}.  
Here we agree that a normal space is, by definition, required to be Hausdorff. This convention will also be in force throughout Section~\ref{s:prop}.

\begin{prop}\label{p:prod}
Let $f\ts f': X\ts X'\to Y \ts Y'$ be the product of two maps $f: X\to Y$ and $f':X'\to Y'$. If $\,Y\times Y'$ is normal, then $\,\Sv(f\ts f')\le \Sv(f) + \Sv(f')$.
\end{prop}

\begin{defin}\label{d:tc1}
Let $X$ be a path-connected space. The {\it $n$-th topological complexity} of $X$, $\TC_n(X)$, is the Schwarz genus of the fibration
\begin{equation}\label{enX}
e_n^X=e_n:X^{J_n}\rightarrow X^n,\quad e_n(\gamma)=(\gamma({1_1}),\ldots ,\gamma({1_n}))
\end{equation}
where $J_n$ is the wedge of $n$ closed intervals $[0,1]$ (each with $0\in[0,1]$ as the base point), and $1_i$ stands for $1$ in the $i^{\mathrm{th}}$ interval. 
\end{defin}

Note that~(\ref{enX}) is the standard fibrational substitute for the iterated diagonal map $d_n=d_n^X\colon X \to X^n$, so $\mathrm{TC}_n(X)=\Sv(d_n^X)$. More generally, for a contractible Êspace $Y_n$ with~$n$ distinct distinguished points $v_1,\ldots, v_n\in Y_n$, consider the evaluation map $e_{Y_n}: X^{Y_n}\to X^n$, $e_{Y_n}(f)=(f(y_1),\ldots,f(y_n))$. Because of the contactibility of $Y_n$, the genus of $e_{Y_n}$ is equal to $\TC_n(X)$, the proof is just as the one in Ê\cite[Remark~3.2.5]{R}. In particular, we can take $Y_n$ to be a tree with $n$ leaves, or the unit interval $I_n$, say with distinguished points $v_i=(i-1)/(n-1)$, $i=1\ldots,n$. In the latter case we see that the $n$-th higher topological complexity gives a topological measure of the complexity of the motion planning problem where the robot is required to visit $n$ ordered prescribed stages. For this reason, we also refer to $\TC_n$ as the $n$-th sequential topological complexity. Farber's $\TC$ is $\TC_2+1$.

\m
Other fibrations (which not necessarily give fibrational substitutes of the iterated diagonal) can be used to define $\TC_n$. Indeed, let $G_n$ be any connected graph where $n$ ordered distinct vertices $v_1,\ldots,v_n$ have been selected. We assert that the evaluation map $e_{G_n}:X^{G_n}\to X^n$ at the chosen vertices has $\Sv(e_{G_n})=\TC_n(X)$. To see this, choose  maps $I_n\to G_n\to J_n$ preserving the selected vertices. For instance, the latter map can be taken so to collapse most of $G_n$ to the base point in $J_n$, except that the first half of each directed edge $(v_i,v)$ in $G_n$ is mapped linearly onto the directed edge $(1_i,0)$ in $J_n$ (in particular vertices $v_i$ are mapped to vertices~$1_i$). Since the induced maps $X^{J_n}\to X^{G_n}\to X^{I_n}$ are compatible with the three evaluation maps, we get $\Sv(e_{I_n})\leq\Sv(e_{G_n})\leq\Sv(e_{J_n})$. But, as explained in the paragraph above, the extremes in the previous chain of inequalities agree with $\TC_n(X)$.

\m
We close this section setting notation relevant to the construction (in~\secref{s:bounding}) of our two symmetric versions of higher topological complexity.

\m
The action of the symmetric group $\Sigma_n$ on $\{1_1,\ldots,1_n\}$ extends to one on $J_n$. This yields corresponding $\Sigma_n$-actions on $X^n$ and $X^{J_n}$ in such a way that~(\ref{enX}) is an equivariant map. The action is free on the configuration space $\CF_n(X)$ of $n$ ordered distinct points in $X$ and, consequently, on $e_n^{-1}(\CF_n(X))$. Thus, at the level of orbit spaces we get a fibration $$\varepsilon_n^X=\varepsilon_n\colon Y_n(X)\to\B_n(X)$$ where $Y_n(X)=e_n^{-1}(\CF_n(X))/\Sigma_n$ and $\B_n(X)=\CF_n(X)/\Sigma_n$---the latter being the usual ``braid'' configuration space of cardinality-$n$ subsets of $X$. 

\m
We think of $\Sv(\varepsilon_n^X)$ as giving a measure for the topological complexity of the {\it $n$-th ubiquitous} motion planning problem on $X$. This concept serves in \secref{s:sym} as the building block relating our two symmetrized forms of $\TC_n$, see Theorem~\ref{generalinequality} and Definition~\ref{d:SStcn}.  \secref{s:bounding} will be devoted to exploring $\Sv(\varepsilon_n^{{S^k}})$.

\m
Note that the commutative diagram (where horizontal arrows are canonical projections)
\begin{equation}\label{lases}
\CD
e_n^{-1}(\CF_n(X)) @>>> Y_n(X)\\
@Ve_nVV @VV\varepsilon_nV\\
\CF_n(X) @>>> \B_n(X)
\endCD
\end{equation}
is a pull-back square, so that (local) sections of $\varepsilon_n$ correspond to $\Sigma_n$-equivariant (local) section of $e_n$. In particular, the homotopy fiber of $\varepsilon_n$ is $(\Omega X)^{n-1}$, just as for $e_n$ (\cite[Remark 3.2.3]{R}). For instance, a copy of $(\Omega X)^{n-1}$ sits inside the fiber of $e_{n}$ over an $n$-tuple $(x_1,x_2,\ldots,x_n)$ as the strong deformation retract consisting of {\it multipaths} $\{\gamma_j\}_{j=1}^n$ for which $\gamma_1$ is the constant path at $x_1$. Here and below the term ``multipath'' refers to an element $\gamma\in X^{J_n}$, and we will use the notation $\gamma=\{\gamma_j\}_{j=1}^{n}$ where $\gamma_j$ is the restriction of $\gamma$ to the $j$-th wedge summand of $J_n$.

\section{Properties of higher topological complexity}\label{s:prop}

The higher topological complexities of a space $X$ {are} closely related to the category of cartesian powers of $X$. The first indication of such a property comes from the inequality
\begin{equation}\label{eq:tccat}
\TC_n(X)\le \cat(X^n)
\end{equation}
which is an immediate consequence of the well known fact that the Schwarz genus of a fibration does not exceed the category of the base space. On the other hand, the inequality $\cat(X)\le\TC_2(X)$ is well known, and can be generalized to:

\begin{prop}\label{p:lowerbound}
For any path-connected space $X$,
$$
\cat(X^{n-1})\le \mathrm{TC}_n(X).
$$
\end{prop}
\begin{proof}
Let $\TC_n(X)=k$ and choose a covering $B_0\cup B_1\cup \cdots \cup
B_k=X^n$ such that there is a continuous section $s_i$ for $e_n^{X}$ over $B_i$ for $i=0,\ldots,k$. Let $p:X^n\rightarrow X$ be the projection onto the first factor, choose $x_1\in X$, and put $A_i=p^{-1}(x_1)\cap B_i$. Note that $\{A_i\}_{i=0}^k$ is an open cover for $p^{-1}(x_1)$. Since {$p^{-1}(x_1)$} is homeomorphic to $X^{n-1}$, it suffices to show that each $A_i$ is contractible within $p^{-1}(x_1)$.

\smallskip
For a point $(x_1, x_2, \ldots, x_n)\in A_i$ consider the $n$ paths $\gamma_1,\ldots,\gamma_n$ making up the multipath $s_i(x_1,x_2,\ldots,x_n)=\{\gamma_j\}_{j=1}^{n}$. Then $\gamma_j(1)=x_j$ and $\gamma_j(0)=x_0$ for some $x_0\in X$ {which is independent of $j\in\{1,\ldots n\}$.} Then, the constant path $\delta_1$ at $x_1$, and the paths $\delta_j$ ($j=2,\ldots,n$)---formed by using {the time reversed path} $\gamma_j^{-1}$ the first half of the time, and $\gamma_1$ the second {half---are} the components of a path $\delta=(\delta_1,\ldots,\delta_n)$ in $p^{-1}(x_1)$ from $\delta(0)=(x_1,x_2,\ldots, x_n)$ to $\delta(1)=(x_1,x_1,\ldots,x_1)$. The continuity of $s_i$ implies that $\delta$ depends continuously on $(x_1,x_2,\ldots, x_n)$, {so we have constructed} a contraction of $A_i$ to $(x_1,x_1,\ldots,x_1)$ in $p^{-1}(x_1)$. Thus,
$\cat(X^{n-1})\le\TC_n(X)$.
\end{proof}

\begin{rem}
Using the fact that $\cat(X^n)\geq n$ {if} $X$ is not contractible
(\cite[Theorem 1.47]{CLOT}), we see that \propref{p:lowerbound} recovers \cite[Proposition 3.5]{R}.
\end{rem}

\propref{p:lowerbound} and \eqref{eq:tccat} yield:
\begin{cor}\label{c:bounds}
For any path-connected space $X$,
$$
\cat(X^{n-1})\le\TC_n(X)\le \cat(X^n).
$$
\end{cor}

We next show that the lower bound in Corollary~\ref{c:bounds} is optimal for topological groups.

\begin{prop}\label{p:topologicalgroup}
For any path-connected topological group $G$,
$$
\mathrm{TC}_n(G)\le  \cat(G^{n-1}).
$$
\end{prop}
\begin{proof}
Let $\epsilon$ denote the neutral element of $G$. Let $k=\cat(G^{n-1})$ and choose an open covering $A_0\cup\,\cdots\,\cup A_k=G^{n-1}$ where each $A_i$ {($i\in\{0,\ldots,k\}$)} contracts in $G^{n-1}$ to an $(n-1)$-tuple $p_i$. Since $G$ is path-connected, each contracting homotopy can be extended {as} to arrange that $p_i=(\epsilon,\ldots,\epsilon)=\epsilon^{(n-1)}$ for all $i=0,\ldots,k$. 

\smallskip
Then, for $i \in \{0, \ldots, k \}$ set
$$B_i=\{ (g, g{a}_2,\ldots,g{a}_n)\,|\,({a}_2,\ldots,{a}_n)
\in A_i, \,g\in G\},$$ which is open in $G^n$.
We assert that {$e_n^G$} admits a (continuous) section over each
$B_i$. Indeed, {for each $i$} the contractibility of $A_i$ in $G^{n-1}$ yields 
a path $\gamma_a$
in $G^{n-1}$ joining $\epsilon^{(n-1)}$ to each $a=(a_2, \ldots, a_n)\in
A_{i}\subset G^{n-1}$ and depending continuously on $a \in A_i$. Augment $\gamma_a$ to a path $\gamma_a'$ from $\epsilon^{(n)}$ to $(\epsilon, a_2, \ldots, a_n) \in B_i$ {with} the first coordinate remaining constant.  Then, for any $g\in G$, $g \gamma_a'$ is a path joining $(g,\ldots,g)=g\epsilon^{(n)}\in G^n$ to $(g, ga_2, \ldots, ga_n)\in B_i$ and depending continuously on $n$-tuples in $B_i$. Then, we get the required section  
$$
s_{i}: B_i\to G^{J_n}
$$
where, on the $j^{\mathrm{th}}$ interval of $J_n$, $s_i(g, ga_2,
\ldots, ga_n)$ is the $j^{\mathrm{th}}$ coordinate of $g\gamma_a'$.

\smallskip
The proof will be complete once we check that
$B_0\cup \,\cdots\, \cup B_k=G^n$. Take $(b_1,
\ldots, b_n)\in G^n$ and put $g=b_1$ and $a_i=g^{-1}b_i$. Then there
exists $j$ such that $(a_2, \ldots, a_n)\in A_j$. So, $(b_1,
\ldots, b_n)\in B_j$.
\end{proof}

\corref{c:bounds} and \propref{p:topologicalgroup} combined yield:
\begin{theorem}\label{t:topologicalgroup}
For any path-connected topological group $G$,
$$
\mathrm{TC}_n(G)= \cat(G^{n-1}).
$$
\end{theorem}

Alternatively, we can look at the growth of $\TC_n$ in terms of the difference of any two consecutive values of $n$.

\begin{cor}\label{c:tgroupbounds}
Let $G$ be a path-connected topological group all of whose finite cartesian powers $G^k$ are normal\hspace{.7mm}\footnote{As noted in Section~\ref{s:prel}, we assume that a normal space is, by definition, Hausdorff. Thus, in view of the classical Birkhoff-Kakutani theorem, the normality hypothesis in \corref{c:tgroupbounds} holds when $G$ satisfies the first axiom of countability---i.e. provided $G$ is metrizable.}. Then for $n\ge3$,
$$
\TC_n(G)-\TC_{n-1}(G)\le \cat (G).
$$
\end{cor}
\begin{proof}
This is a consequence of \theoref{t:topologicalgroup} and the product inequality for the category---valid under the current normality assumptions in view of \propref{p:prod}.
\end{proof}

Unlike {with topological} groups, higher topological complexities of {an arbitrary} path-connected space $X$ do not {appear} to be completely determined by the category of cartesian powers of $X$. Nonetheless, we can directly obtain the following bound on the difference of {two} consecutive higher topological complexities of $X$.

\begin{prop}\label{c:difbounds}
Let $X$ be a path-connected space all of whose finite cartesian powers $X^k$ are normal. Then for $n\ge3$,
$$
\TC_n(X)-\TC_{n-1}(X)\le \cat (X^2).
$$
\end{prop}
\begin{proof}
Use the argument in the proof of \corref{c:tgroupbounds}, replacing \theoref{t:topologicalgroup} by the inequalities in \corref{c:bounds}.
\end{proof}

{In particular $\TC_n(X)$ is bounded from above by a linear function on $n$ with slope $\cat(X^2)$.} According to \cite[(5.1)]{R}, this slope can be improved to $\TC_2(X)$.

\m
Next we consider the higher analogue of the usual cup-length lower bound for $\TC$. Recall $d_n=d_n^X\colon X\to X^n$ stands for the iterated diagonal map. In the following definition we allow cohomology with local coefficients.

\begin{defin} Given a space $X$ and a positive integer $n$, $\cl(X,n)$ denotes the cup-length of elements in the kernel of the map induced in cohomology by $d_n^X$. Thus, $\cl(X,n)$ is the largest integer $m$ for which there exist cohomology classes $u_i\in H^*(X^n; A_i)$ such that $d_n^*u_i=0$ for $i=1, \ldots, m$ and
\[
u_1\cupr \cdots \cupr u_m\ne 0\in H^*(X^n;A_1\otimes \cdots \otimes A_m).
\]
\end{defin}

The following result, which follows directly from \cite[Theorems~4 and~5']{Sv}, bounds $\TC_n(X)$ from below by $\cl(X,n)$, and from above by a ratio between the connectivity $\conn(X)$ and homotopy dimension $\hdim(X)$ of $X$---the latter being the smallest dimension of CW complexes having the homotopy type of $X$.

\begin{thm}\label{t:clest}
For any path-connected space $X$ we have the inequalities $$\cl(X,n)\le \TC_n(X)\leq\frac{n\hdim(X)}{\conn(X)+1}.$$
\end{thm}

We will also need the following bound on  $\cl(X\ts S^k, n)$ in terms of  $\cl(X,n)$.

\begin{thm}\label{t:multbysphere}
For any path-connected space $X$ and {positive integers} $n$ and $k$
we have $\cl(X\ts S^k,n)\ge \cl(X,n)+n-1$. This inequality can be improved
to $\,\cl(X\ts S^k,n)\ge \cl(X,n)+n$ provided $k$ is even and $H^*(X)$ is
torsion-free.
\end{thm}
\begin{proof}
Let $v$ be a generator of $H^k(S^k)=\Z$. Let $p_i: (S^k)^n\to
S^k$ be the projection onto the $i\th$ factor and put
$v_i={p_i^*(v)}$ for $i=1,\ldots, n$.
Assume that $\cl(X,n)=m$ and take $u_1, \ldots, u_m$ such
that $d_n^*{(u_j)}=0$ for $j=1,\ldots, m$ and $u_1\cupr\cdots \cupr u_m\ne 0$.

\smallskip
To prove the first assertion note that $d_n^*(v_i-v_1)=0$ for $i>1$, while
the basis element $v_2\cupr\cdots\cupr v_n\in H^*\left((S^k)^n\right)$
appears in the reduced expansion (using distributivity) of
$(v_2-v_1)\cupr\cdots\cupr(v_n-v_1)$. Hence,
\[ u_1\cupr\cdots \cupr u_m\cupr (v_2-v_1)\cupr \cdots\cupr (v_n-v_1)\ne 0.
\]
Thus $\cl(X\ts S^k, n)\ge \cl(X,n)+n-1$.

\smallskip
Assume now that $k$ is even and that $H^*(X)$ is torsion-free. The element
$v_1+v_2+\cdots +v_{n-1}-(n-1)v_n$ lies in the kernel of $d_n^*$ and has
cup $n\th$ power equal to a non-zero multiple of
$v_1 \cupr v_2\cupr\cdots\cupr v_n$. Hence,
\[
u_1\cupr\cdots \cupr u_m\cupr (v_1+v_2+\cdots +v_{n-1}-(n-1)v_n)^n\ne 0.
\]
Thus $\cl(X\ts S^k, n)\ge \cl(X,n)+n$.
\end{proof}

In \cite{F1} Farber obtained the subadditivity of $\TC_2$ under suitable topological hypothesis. The corresponding property for higher topological complexity is given next.

\begin{prop}\label{p:product}
Let $X$ and $Y$ be path-connected spaces. If $(X\times Y)^n$ is normal, then
$
\TC_n(X\times Y)\le  \TC_n(X)+\TC_n(Y).
$
\end{prop}
\begin{proof}
The natural homeomorphisms
\begin{equation*}\begin{aligned}
(X\ts Y)^n &\to X^n\ts Y^n,\\
((x_1, y_1), \ldots, (x_n,y_n))&\mapsto (x_1, \ldots, x_n,y_1, \ldots, y_n),
\;\;x_i\in X, \;\;y_j \in Y
\end{aligned}\end{equation*}
and
\begin{equation*}
\CD
\begin{aligned}
(X\ts Y)^{J_n}&\to X^{J_n}\ts Y^{J_n},\\
 (\gf: J_n\to X\ts Y)&\mapsto ((p_X\circ\gf: J_n \to X), (p_Y\circ\gf: J_n \to Y))
\end{aligned}
\endCD
\end{equation*}
fit into {the} commutative diagram
\[
\CD
(X\ts Y)^{J_n}@>>> X^{J_n}\ts Y^{J_n}\\
@Ve_n^{X\ts Y}VV @VVe_n^X \ts e_n^Y V\\
(X\ts Y)^n @>>> X^n\ts Y^n.
\endCD
\]
So, the desired conclusion follows directly from \propref{p:prod}.
\end{proof}

As revealed in the case of spheres (next), \propref{p:product} is optimal in general.

\begin{cor}\label{c:spheres}
$\TC_n(S^{k_1}\times S^{k_2}\times \cdots \times S^{k_m})=m(n-1)+l$ where $l$ is the number of even dimensional spheres.
\end{cor}
\begin{proof}
Note that $\TC_n(S^k)=\cl(S^k,n)$ for all $k$,~\cite[Section 4]{R}.
Then the inequality $\cl(S^{k_1}\times \cdots \times S^{k_m}, n)\ge m(n-1)+l$ follows from \theoref{t:multbysphere} by induction, so $\TC_n(S^{k_1}\times \cdots \times S^{k_m})\ge m(n-1)+l$ by \theoref{t:clest}.
The opposite estimate follows from \propref{p:product}.
\end{proof}

The calculation of the higher topological
complexity of the $k$-dimensional torus $T^k=(S^1)^k$,
partially solved for $k=2$ in \cite[Proposition 5.1]{R},
is a consequence of either \corref{c:spheres} or
\theoref{t:topologicalgroup}.

\begin{cor}\label{c:torus}
$\TC_n(T^k)=k(n-1)$.
\end{cor}

\begin{thm}\label{t:cohom}
Let $X$ be a CW complex of finite type, and $R$ a principal
ideal domain. Take $u\in H^d(X;R)$
with $d>0$, $d$ even, and assume that
the $n$-fold iterated self $R$-tensor product
$u^m\otimes\cdots\otimes u^m\in(H^{md}(X;R))^{\otimes n}$
is an element of infinite additive order. Then $\TC_n(X)\ge mn$.
\end{thm}
\begin{proof}
{For $i=1,\ldots,n$,} let $p_i: X^n\to X$ be the projection onto the $i\th$ factor and put $u_i=p_i^*{(u)}\in H^d(X^n;R)$.
In view of~\theoref{t:clest}, the required inequality follows from
\begin{equation}\label{explain}
v:=(u_2-u_1)^{2m}(u_3-u_1)^m\cdots(u_n-u_1)^m \ne 0.
\end{equation}
In order to check~(\ref{explain}), note that $v$ comes from the tensor {product, which} injects into the cohomology of the cartesian product
by the K\"unneth Theorem (this is where the finiteness hypotheses
are used).
So, calculations {can be} performed in the former $R$-module.
Now, assuming that $\dim(X) \le dm+1$, we have
\begin{eqnarray*}
v &=& (u_2-u_1)^{2m}(u_3-u_1)^m\cdots(u_n-u_1)^m\\
  &=&  (-1)^m\binom{2m}{m}u_1^mu_2^m(u_3-u_1)^m\cdots(u_n-u_1)^m\\
  &=& (-1)^m\binom{2m}{m}u_1^mu_2^mu_3^m(u_4-u_1)^m\cdots(u_n-u_1)^m\\
  &=&  \cdots\\
  &=&  (-1)^m\binom{2m}{m}u_1^mu_2^m\cdots u_n^m,
\end{eqnarray*}
which is non-zero by hypothesis. On the other hand,
for $\dim(X)$ arbitrary, consider the skeletal inclusion
$j: X^{(dm+1)}\to X$ and note that $v\ne 0$ since $j^*{(v)}\ne 0$.
\end{proof}

\begin{cor}\label{c:symplec}
For every closed simply connected symplectic manifold $M^{2m}$ we have $\TC_n(M)=nm$.
\end{cor}
\begin{proof}
This follows {from}
\theoref{t:cohom} (taking $u$ to be the cohomology class
given by the symplectic 2-form on $M$, {and noting} that the hypothesis on
$u^m\otimes\cdots\otimes u^m$ holds since {the} coefficients are taken over the reals), {inequality} \eqref{eq:tccat}, the product inequality {for category}, and the inequality $\cat (M^{2m})\le m$ which follows from \cite[Theorem~5, page~75]{Sv}. (This argument also yields $\cat (M^{2m})=m$, a well known fact.)
\end{proof}

Of course, \corref{c:symplec} applies to complex projective spaces.
In the quaternionic case essentially the same proof gives:

\begin{cor}\label{c:projective} The quaternionic projective space of
real dimension $4m$, $\mathbb{H}\P^m$, has $\TC_n(\mathbb{H}\P^m)=nm$.
\end{cor}

Note that Corollaries~\ref{c:symplec} and~\ref{c:projective} imply that the upper bound in Corollary~\ref{c:bounds} as well as both bounds in \theoref{t:clest} are optimal in general.

\section{Symmetric topological complexity}\label{s:sym}

In this section we introduce two symmetric versions of $\TC_n$.
One of them, $\STC_n$, has the advantage of being a homotopy
invariant. The other, $\TC_n^S$, gives (up to our normalization convention) the natural generalization of the symmetric topological complexity studied by Farber and Grant in \cite{FGsym}. We begin with the $n=2$ case of the homotopically well-behaved version.

\m
Consider the involutions $\tau: X\ts X \to X\ts X$ and $\ov \tau:
P(X) \to P(X)$ defined by $\tau(x,y)=(y,x)$ and
$\ov\tau(\gamma)(t)=\gamma(1-t)$, for $(x,y)\in X\times X$ and $\gamma\in P(X)$. We work with symmetric subsets $A\subseteq X\ts X$ (i.e.~those for which $\tau A=A$), and equivariant maps $s:A\to P(X)$ (i.e.~those satisfying $\ov\tau(s(a))= s(\tau(a))$ for all $a\in A$).

\begin{defin}\label{d:stc}
$\STC(X)$ is the least integer $k$ such that $X\ts
X=A_0\cup A_1\cup\,\cdots\,\cup A_k$ where each $A_i$ is open,
symmetric, and admits a continuous equivariant section
$s_i: A_i\to P(X)$ of the map $e_2$ in \eqref{enX}.
\end{defin}

Before proving (in \propref{p:hinv} below)
that $\STC(X)$ is a homotopy invariant of $X$, we show that its numerical value differs by at most one from the numerical value of Fraber-Grant's symmetric topological complexity. In accordance with the normalization hypothesis in this paper, we must compare {$\STC(X)$} with 
\begin{equation}\label{eq:tcSfarber}
\TC_2^S(X)=\Sv(\eps_2)+1
\end{equation}
where $\eps_2$ is the map on the right hand side of \eqref{lases}. Note that, under the perspective of~\cite{FGsym}, the ``$+1$'' summand in~(\ref{eq:tcSfarber}) is meant to take into account the obvious equivariant section of $e_2$ on the diagonal.

\begin{prop}\label{p:compare}
For each ENR $X$ we have
\[
\TC_2^S(X)-1\le \STC(X)\le\TC_2^S(X).
\]
\end{prop}

\begin{remark}\label{greatergenerality}
We will prove a more general version of \propref{p:compare}
(\theoref{generalinequality} below). The proof of the general version is considerably more elaborate as it
requires a rather involved use of {\it equivariant} euclidean
neighborhood retracts. For the sake of clarity, we offer
first the much simpler argument proving \propref{p:compare}.
\end{remark}

\begin{proof}[Proof of Proposition~{\em\ref{p:compare}}] 
To prove the first inequality, take an open covering
$X\ts X=A_0\hspace{.2mm}\cup\cdots\cup\hspace{.2mm} 
A_k$ where each $A_i$ is symmetric and
has {a continuous} equivariant section of $e_2$.
The $\Z/2$-action $\tau$ on $X\ts X$
yields the orbit map ${\rho_2}: X\ts X\to (X\ts X)/\tau$. Then, {for} each {$i=0,\ldots,k$,} 
${\rho_2}(A_i-d_2(X))$ is open and has a section of $\eps_2$, 
and thus $\Sv(\eps_2)\le \STC(X)$.

\smallskip
For the second inequality, take $B_0, \ldots, B_l$, with $B_0\cup\cdots\cup
B_l={\rho_2}(X\ts X-d_2(X))$ where each $B_i$ is open and has a
section of $\eps_2$. Then each ${\rho_2}^{-1}(B_i)$ is symmetric, open in
$X\times X$, and admits an equivariant section of $e_2$,
cf.~\cite[Lemma 8]{FGsym}.
Further, since $X$ is an ENR, there is a symmetric open
neighborhood of $d_2(X)$
supporting an equivariant section of $e_2$ (see the proof
of~\cite[Corollary 9]{FGsym}). Consequently $\STC(X)\le 1+\Sv(\eps_2)$.
\end{proof}

The two examples below show that both bounds in \propref{p:compare} are optimal in general.

\begin{ex}\label{e:sigma}
For $X$ contractible we have
$\TC_2(X)=\STC(X)=0$ while $\TC^S_2(X)=1$.
Indeed, take a point $x_0\in X$ and a contraction $H: X\ts I \to X$, with
$H(x,0)=x$ and $H(x,1)=x_0$ for all $x\in X$. Given {$(a,b)\in X\times X$,} take the
path ${\sigma=s(a,b)}: I \to X$ such that ${\sigma}(t)=H(a,2t)$ for $0\le t\le 1/2$ and
${\sigma}(t)=H(b,2-2t)$ {for $1/2\le t\le 1$.} Then $s$ is an equivariant
section for $e_2^X$ and, in view of
the general inequality
$$
\TC_2(X)\le\STC(X),
$$
this gives $\TC_2(X)=\STC(X)=0$. The same argument, but now using
\eqref{eq:tcSfarber}, gives $\TC^S_2(X)=1$ (see \cite[Example 7]{FG}).
\end{ex}

\begin{ex}\label{e:sigmasphere}
The numbers $\TC_2^S(S^k)$ and $\TC_2(S^k)$ have been computed in~\cite[Corollary~18]{FGsym} and~\cite{F1}, respectively. Here we use the  inequalities $\TC_2\le\STC\le\TC^S$ together with the fact that $\TC_2^S(S^k)=2=\TC_2(S^{2k})$ to deduce $\TC^\Sigma(S^{2k}) = \TC^S_2(S^{2k})=2$ for all $k$. On the other hand, since
$\TC_2(S^{2k+1})=1$, the above argument only gives $1\leq\TC^\Sigma(S^{2k+1}) \leq \TC^S_2(S^{2k+1})=2$. Incidentally, note that the
construction in~\cite[Example 4.8]{F3} gives an open covering
$S^{2k+1}\times S^{2k+1}=A_0\cup A_1$ by symmetric sets $A_i$, and continuous sections of $e_2$ over each $A_i$, {$i=0,1$.} However, one of these sections is not equivariant, which prevents us from deducing $\STC(S^k)=1$.
\end{ex}

We next define higher analogues of $\STC$. Recall that for a given $n$, the symmetric group $\Si_n$ acts on $X^n$ {and} on $X^{J_n}$ by permuting coordinates and paths, respectively. Further, the fibration $e_n$ in \eqref{enX} is $\Si_n$-equivariant. We now work with symmetric subsets $A\subseteq X^n$ (i.e.~those for which $\sigma A=A$ for all $\sigma\in\Sigma_n$), and equivariant maps $s:A\to X^{J_n}$ (i.e.~ those satisfying $\sigma(s(a))= s(\sigma(a))$ for all $a\in A$ and
$\sigma\in\Sigma_n$). \defref{d:stc} can now be extended to:

\begin{defin}\label{d:stcn}
$\STC_n(X)$ is the least integer $k$ such that $X^n=A_0\hspace{.2mm}\cup
A_1\cup\,\cdots\,\cup A_k$ where each $A_i$ is open, symmetric and admits a
continuous equivariant section $s_i: A_i\to X^{J_n}$ for the map $e_n$ in~(\ref{enX}).
\end{defin}

\begin{prop}\label{p:hinv}
$\STC_n(X)$ is a homotopy invariant of $X$.
\end{prop}
\begin{proof}
It suffices to prove that, given $f: Y \to X$ and $g: X \to Y$ with $gf \simeq 1_Y$, we have $\STC_n(X) \ge \STC_n(Y)$ for all $n$. Let
$H: 1_Y \simeq gf$ be a homotopy $H: Y \ts [0,1] \to Y$
such that $H(y,0)=y$ and $H(y,1)=gf(y)$.

\smallskip
Let $A$ be an open symmetric subset of $X^n$, and let $s: A \to
X^{J_n}$ be an equivariant section of $e_n^X$ over $A$.
Given $a=(a_1, \ldots, a_n)\in A$, let $s_i(a)$ denote the restriction of
$s(a)\in X^{J_n}$ to the $i\th$ wedge summand of $J_n$ (this is a path in $X$
joining $x_0$ and $a_i$ for some $x_0\in X$ that depends continuously on $a$).
Note that the {equivariance} of $s$ gives
\begin{equation}\label{lasimetria}
s_i(a_{\sigma(1)},\ldots,a_{\sigma(n)})=s_{\sigma(i)}(a_1,\ldots,a_n)
\text{ \ \ for \ } \sigma\in\Sigma_n.
\end{equation}
Take $B:=(f^n)^{-1}(A)$, where $f^n$ stands for the $n$-th cartesian power of $f$,  
and consider the map $s': B \to Y^{J_n}$ which,
at a given $b\in B$ with $f^n(b)=a$, has
$s'_i(b):=\left(g\circ s_i(a)\right)\cdot
\gamma_i$ as its restriction to the $i\th$ wedge summand of $J_n$, where
$\gg_i$ is the path in $Y$ given by
$$
\gg_i(t)=H(b_i, 1-t).
$$
Then, $s'$ is an equivariant continuous section of $e_n^Y$ over $B$ in view of~ \eqref{lasimetria}).

\m
In this setting, if $X=A_0 \cup \cdots \cup A_k$ where each $A_j$ ($j=0,\ldots,k$) is open, symmetric, and admits a continuous equivariant section of $e_n^X$, then $Y=B_0\cup\hspace{.23mm}\cdots\hspace{.23mm}\cup B_k$ where each $B_{j\,}$---defined as above using $A_{j\,}$---is open, symmetric, and admits a continuous equivariant section of $e_n^Y$. Hence, $\STC_n(X)\ge \STC_n(Y)$.
\end{proof}

The following assertion is our higher analogue of \propref{p:compare}.
\begin{thm}\label{generalinequality}
If $X$ is an ENR, and $\varepsilon_n$ is the map on the
right hand side of \eqref{lases}, then
\begin{equation}\label{cotasgenrs}
\Sv(\eps_n)\le \STC_n(X)\le \Sv(\varepsilon_n)+\cdots+\Sv(\varepsilon_2)+n-1.
\end{equation}
\end{thm}

The first inequality in \eqref{cotasgenrs} follows just
as in the proof of Proposition~\ref{p:compare}: If $e_n$
admits an equivariant section over $A\subset X^n$, then $\eps_n$
admits a section over ${\rho_n}(A\cap C_n(X))$ where ${\rho_n}: X^n\to X^n/\Si_n$
stands for the canonical projection. Our efforts will therefore
focus on the second inequality in \eqref{cotasgenrs}, whose proof
requires some preparation.

\begin{defin} A topological space $X$ with an action of a compact Lie group $G$ is called a {\it euclidean neighborhood $G$-retract} {($G$-ENR)} if $X$ can be $G$-equivariantly embedded, as a $G$-equivariant retract of a $G$-symmetric neighborhood {of $X$, into} an orthogonal representation of $G$.
\end{defin}

In what follows we will make implicit use of the following fact: if a $G$-ENR $X$ is $G$-equivariantly embedded in a given orthogonal representation $\mathbb{R}^N$ of $G$, then there exists a $G$-symmetric neighborhood $U$ of $X$ in $\mathbb{R}^N$ and a $G$-equivariant retraction $U \to X$. As noticed at the end of the introduction in \cite{J}, such a property follows by {applying} the equivariant version of the Tietze Theorem (Tietze-Gleason Theorem, \cite{B,G}) {to} the non-equivariant argument in \cite[Proposition and Definition IV.8.5]{D2}.

\m {We shall use the following weaker version} of~\cite[Theorem 2.1]{J}\footnote{Although Jaworowski's theorem was originally set in terms of a combination of the concepts of ANR's and ENR's, for our formulation the reader should keep in mind the fact that any ENR is an ANR {(which is elementary in view of the Tietze Theorem).}}.

\begin{thm}[Jaworowski]\label{t:jaw}
Let $L$ be a finite group acting on an ENR $Z$. Then $Z$ is an $L$-ENR if for every subgroup $G$ of $L$, the fixed point set $Z^G$ is an ENR.
\end{thm}

Next, consider the $\Si_n$-equivariant filtration
\begin{equation}\label{Dfiltracion}
d_n(X)=\DF^{1}(X)\subset\cdots\subset\DF^{n-1}(X)\subset\DF^n(X)=X^n
\end{equation}
where, for $i\in\{1,\ldots,n\}$, $\DF^i(X)$ is the closed set consisting of the $n$-tuples $(x_1,x_2,\ldots,x_n)$ such that the {\em set} $\{x_1,x_2,\ldots,x_n\}$ has cardinality at most~$i$. {For instance, $\DF^{n-1}(X)$ is the so called fat diagonal in $X^n$, otherwise denoted by $\Delta_n(X)$. Compare the filtration in~(\ref{Dfiltracion}) with the one considered at the end of Section 1 in \cite{kallel}.}

\m
Set $\DF^0(X)=\empt$, and for $1\leq i\leq n$ let $C^i$ stand for the difference $\DF^i(X)-\DF^{i-1}(X)$, the subspace of $n$-tuples $(x_1,x_2,\ldots,x_n)$ such that the set $\{x_1,x_2,\ldots,x_n\}$ has cardinality~$i$. Note that $C^n=\CF_n(X)$ {and that} for $i<n$, each partition $\mathcal{P}=\{P_1,\ldots,P_i\}$ of $\{1,2,\ldots,n\}$ into $i$ nonempty sets determines a closed subspace $C^i_{\mathcal{P}}\subset C^i$ formed by those tuples $(x_1,\ldots,x_n)$ in $C^i$ satisfying $x_r=x_s$ whenever both $r$ and $s$ lie in the same part $P_j$ for some $j$.

\m
Note that $C^i$ is the disjoint union of the {$C^i_{\mathcal{P}}$'s}, each of which maps homeomorphically {onto $\CF_i(X)$ under a suitable coordinate projection.} [For instance, for $n=3$ the three closed subspaces partitioning $C^2$ are {determined by} the three requirements $x_1=x_2$, $x_1=x_3$, and $x_2=x_3$; in the latter case, the required projection can be chosen to be $(x_1,x_2,x_3)\mapsto(x_1,x_2)$.] {Therefore,} we have a continuous (surjective) map $\pi_i\colon C^i\to C_i(X)$.

\m Let $P^i$ denote the subspace of $e_n^{-1}(C^i)$
consisting of those multipaths $\ga=\{\ga_i\}_{i=1}^{n}$ satisfying
$\ga_k=\ga_\ell$ whenever $\ga_k(1_k)=\ga_\ell(1_\ell)$. Proceeding as above, we get a continuous surjection $\Pi_i\colon P^i\to e_i^{-1}(\CF_i(X))$ in such a way that in the following commutative diagram
\begin{equation}\label{commdiagPllBck}
\CD
X^{J_n} @<<< P^i @>\Pi_i>> e_i^{-1}(\CF_i(X)) @>>> Y_i(X)\\
@Ve_nVV @Ve_nVV @Ve_iVV @VV\varepsilon_iV\\
X^n @<<< C^i @>\pi_i>> \CF_i(X) @>>> \B_i(X)
\endCD
\end{equation}
the second and third squares are pullbacks, {and the two left-most horizontal
maps are inclusions but do not determine a pullback square.}

\m Our last ingredient in preparation for the proof of~(\ref{cotasgenrs}) is given by taking an arbitrary open subset $W$ of $\B_i(X)$. We {then} let $A=\pi_i^{-1}(W')$ where $W'$ stands for the inverse image of $W$ under the {projection} $\CF_i(X)\to\B_i(X)$. Clearly $W'$ is $\Sigma_i$-symmetric and $A$ is $\Sigma_n$-symmetric. This setup will be in force in the following two auxiliary results, {which are} the basis of our proof of the second inequality in~{(\ref{cotasgenrs}).}

\begin{lemma}\label{l:aux1}
The space $A$ is a $\Si_n$-ENR.
\end{lemma}
\begin{proof}
Note first that every $C^i_{\cp}$ is an ENR, because it is homeomorphic to $C_i(X)$ which, in turn, is an open subset of the ENR $X^i$. Now, every $g\in \Si_n$ yields a homeomorphism from any given $C^i_{\cp}$ onto some $C^i_{\cp'}$. {In particular for $\cp=\cp'$, if there is some $x\in C^i_{\cp}$ fixed by $g$, then} $g\cdot y=y$ for {all} $y\in C^i_{\cp}$, i.e.~$(C^i_{\cp})^g=C^i_{\cp}$. Hence, for any subgroup $G$ of $\Si_n$, the set  $(C^i_{\cp})^G$ is either empty or the whole of $ C^i_{\cp}$, and therefore an ENR. {Consequently}, $(C^i)^G$ is an ENR since $C^i$ is the disjoint union of the various $C^i_{\cp}$'s, {and} $A^G$ is an ENR since $A$ is open in $C^i$. Thus, by \theoref{t:jaw}, $A$ is a $\Si_n$-ENR, as asserted.
\end{proof}

\begin{lemma}\label{l:aux2}
Assume $s: A\to P^i$ is a $\Si_n$-equivariant section of the second vertical map in~\emph{(\ref{commdiagPllBck})}. Then there is a $\Si_n$-symmetric neighborhood $U$ of $A$ in $X^n$ that admits a $\Si_n$-equivariant section $\sigma\colon U\to X^{J_n}$ of the first vertical map in~\emph{(\ref{commdiagPllBck})}.
\end{lemma}
\begin{proof}
We begin by noticing that, as a consequence of \theoref{t:jaw}, $X^n$ is a $\Si_n$-ENR. {Indeed,} for any subgroup $G$ of $\Si_n$, the fixed point set of $G$ on $X^n$ is an intersection of  hyperplanes $x_i=x_j$ in $X^n$. {Hence, $\left(X^n\right)^G$ is an ENR since it is
homeomorphic to $X^m$ for $m\leq n$.} Thus, we can take $\Si_n$-equivariant embeddings $A\to X^n \to \R^N$, and a $\Si_n$-equivariant retraction $r'\colon O\to A$ of a $\Si_n$-symmetric neighborhood $O$ of $A$ in $\R^N$, where $\R^N$ is an orthogonal representation of $\Sigma_n$.

\smallskip
{Set} $V=O\cap X^n$. Then $V$ is a $\Si_n$-symmetric neighborhood of $A$ in $X^n$, and $r={r'|_{V}}:V \to A$ is a $\Si_n$-equivariant retraction. Note that $V$ is an open $\Si_n$-symmetric subset of the $\Si_n$-ENR $X^n$, and so $V$ is a $\Si_n$-ENR {too}. {We can then} choose an open $\Si_n$-symmetric neighborhood $Y$ of $V$ in $\R^N$, and a $\Si_n$-equivariant retraction $\gr: Y \to V$. Let $U\subset V$ consist of all points $v\in V$ such that the segment from $v$ to {$i\circ r(v)$} lies in $Y$ {where} $i$ stands for the inclusion $A\hookrightarrow V$ {(cf.~\cite[Corollary IV.8.7]{D2}).} Clearly $U$ is a neighborhood of $A$ in $V$, and hence in $X^n$. Furthermore, the {composition} {$i\circ r|_{U}$} and the inclusion $U\hookrightarrow V$ are homotopic via the homotopy
\[
\Phi: U\ts I \to V, \quad \Phi(u,t)=\gr\left(t\cdot u+(1-t)\cdot {i\circ r}(u)\right).
\]
Note that $U$ is $\Si_n$-symmetric and $\Phi$ is $\Si_n$-equivariant, since the $\Si_n$-action on $\R^N$ is orthogonal and so {it} maps lines to lines.

\smallskip
We use the homotopy {$\Phi$ in order} to construct a $\Si_n$-equivariant section $\sigma\colon U\to X^{J_n}$ of the first vertical map in~(\ref{commdiagPllBck}). For $x\in U$, consider the path $\gb\colon I\to V$, $\gb(t)=\Phi(x,t)$, starting at $y=\gb(0)=r(x)\in A$ and ending at $x$. Since $V$ {is a subset of} $X^n$, we {can} set $x=(x_1,\ldots, x_n)$, $y=(y_1,\ldots, y_n)$, and $\gb=(\gb_1,\ldots,\gb_n)$, so each $\gb_i$ is a path in $X$ from $y_i$ to $x_i$. Further, $s(y)$ gives a multipath $\{\ga_i\}_{i=1}^{n}$ with $\ga_i(1)=y_i$ and $\ga_i(0)=\ga_j(0)$ for all $1\leq i,j\leq n$. Then the multipath $\{\ga_i\cdot\gb_i\}_{i=1}^n$ determines an element $\sigma(x)\in X^{J_n}$ with $e_n(\sigma(x))=x$. This defines the required $\Si_n$-equivariant section over $U$.
\end{proof}

Note that the two pull-back squares in~(\ref{commdiagPllBck}) imply that the hypothesis in \lemref{l:aux2} holds whenever $W$ {(the arbitrary open subset of $\B_i(X)$ taken in the paragraph previous to \lemref{l:aux1})} is chosen to admit a section of the fourth vertical map in \eqref{commdiagPllBck}. Thus we {obtain the following:}

\begin{proof}[Proof of Theorem~\emph{\ref{generalinequality}} (conclusion)]
In view of Lemmas~\ref{l:aux1} and~\ref{l:aux2} we can choose $1+\Sv(\varepsilon_i)$ $\hspace{.7mm}\Si_n$-equivariant local sections for $e_n$ whose domains cover $C^i$, and thus a total of
\begin{equation}\label{totalsects}
\sum_{i=2}^{n}(1+\Sv(\varepsilon_i))+1=\Sv(\varepsilon_n)+\cdots+\Sv(\varepsilon_2)+n
\end{equation}
$\Si_n$-equivariant local sections for $e_n$ whose domains cover $X^n$. Here the ``$+1$''  on the left-hand side of~(\ref{totalsects}) accounts for the obvious equivariant section on the diagonal $D^1(X)$. The theorem follows.
\end{proof}

A comparison of \propref{p:compare} and \theoref{generalinequality} suggests the following generalization of~\eqref{eq:tcSfarber}:

\begin{defin}\label{d:SStcn}
For $n\geq2$ set $$\TC_n^S(X)=\Sv(\varepsilon_n)+\cdots+
\Sv(\varepsilon_2)+n-1.$$
\end{defin}

This is a minor variation of the one proposed in the short final section in~\cite{R}, and will be explored next for $X$ a sphere.

\section{Schwarz genus of $\varepsilon_n$ and  configuration spaces of spheres}\label{s:bounding}

The following result, which is a specialization of~\cite[Theorem~5, page~75]{Sv} (recalling that $(\Omega X)^{n-1}$ is the homotopy fiber of the map $\varepsilon_n=\varepsilon_n^X\colon Y_n(X)\to\B_n(X)$ in \eqref{lases}), gives a general upper bound for $\Sv(\varepsilon_n)$ analogous to that in Theorem~\ref{t:clest}.

\begin{prop}\label{p:conn}
If $X$ is an $(s-1)$-connected space and $\B_n(X)$ has the homotopy type of a $d$-dimensional CW space, then $\Sv(\varepsilon_n)\leq d/s.$
\end{prop}

For instance, $\Sv(\varepsilon_n^X)=0$ for any contractible space $X$. This generalizes the phenomenon noted in Example~\ref{e:sigma}. Part of the goal of this section is to show that the bound in Proposition~\ref{p:conn} becomes an equality in some concrete situations---other than those noted for a contractible space $X$. Yet, the following considerations are written in conjectural terms; non-conjectural statements start from equation~(\ref{cotainf}) on.

\m
The conjectural inequality in~(\ref{conejo}) is based on Proposition~\ref{p:conn}. To illustrate the idea, start by recalling from \exref{e:sigmasphere} the equality $\TC_2^S(S^k)=2$ valid for any~$k$. Farber and Grant prove that $\TC_2^S(S^k)$ is no greater than 2 {by producing} a symmetric motion planner with two local rules. Their construction makes use of a well-known explicit $\Sigma_2$-equivariant deformation retraction $\CF_2(S^k)\to S^k$ that implies a corresponding homotopy equivalence
\begin{equation}\label{pk}
\B_2(S^k) \simeq \R\P^k.
\end{equation}
Here we note that \propref{p:conn} gives an alternative direct way to deduce the inequality $\TC_2^S(S^k)\leq 2$: all that is needed is the fact that $\hdim(\B_2(S^k))=k$. In order to extend this simple argument for higher $\TC^S_n$ we would need to have a good hold on the homotopy dimension of $\B_n(S^k)$. Remark~\ref{fupo} below provides evidence toward the following:

\begin{conjec}\label{t:mainOLD}
For $n\geq2$ and $k\geq1$, $\hdim(\B_n(S^k))=(k-1)(n-1)+1$.
\end{conjec}

\begin{rem}\label{fupo}
Note that the validness of Conjecture~\ref{t:mainOLD} for $n=2$ follows from~(\ref{pk}). Likewise, the case $k=1$ of Conjecture~\ref{t:mainOLD} is well known: $\B_n(S^1)$ has the homotopy type of $S^1$ (cf.~\cite[Proposition 2.5]{kallel}). On the other hand, from the calculations of homology groups in~\cite{FZ}, it can be proved that Conjecture~\ref{t:mainOLD} is true if $\B_n(S^k)$ is replaced by $\CF_n(S^k)$ when $k\geq 3$. At any rate, since  the homotopy dimension of a space is not less than the homotopy dimension of any of its covering spaces, we have
$$
\mbox{hdim}(\B_n(S^k))\ge\mbox{hdim}(\CF_n(S^k))=(k-1)(n-1)+1.
$$
Therefore the crux of the matter in settling Conjecture~\ref{t:mainOLD} (and, as a consequence, the equality $\hdim(\CF_n(S^k))=(n-1)(k-1)+1$) rests in producing a CW~complex of dimension $(k-1)(n-1)+1$ which has the $\Sigma_n$-equivariant homotopy type of $\CF_n(S^k)$. The second and fourth authors of this paper have an ongoing project aiming at such a goal; the basic ideas have been presented in the second half of~\cite{1009.1851v5}. However, it turns out that those ideas require an important tuning and have actually become a completely independent paper (which will appear elsewhere). The present paper then focuses on the first half of~\cite{1009.1851v5}, i.e.~the development of the properties of the sequential topological complexity.
\end{rem}

We have mentioned that the validness of~(\ref{conejo}) would follow from Conjecture~\ref{t:mainOLD}. In fact, in view of Proposition~\ref{p:conn}, we see that  Conjecture~\ref{t:mainOLD} would actually imply the validness of the more detailed but still conjectural estimate:
\begin{equation}\label{c:TCnSpherek}
\mbox{\emph{$\Sv(\varepsilon_i)\leq i-1-\frac{i-2}k,\;$ for $X=S^k$ and $i\geq2$.}}
\end{equation}
The remainder of the section is devoted to presenting evidence for the validness and general optimality of~(\ref{c:TCnSpherek}). 

\m
We have observed that~(\ref{c:TCnSpherek}) holds true for $i=2$. As for its optimality, it is worth observing that Farber and Grant prove in \cite[Section~3]{FGsym} the inequality
\begin{equation}\label{cotainf}
\TC_2^S(S^k)\geq2
\end{equation}
by means of an involved extension of Haefliger's calculation of the mod~$2$ cohomology ring $H^*(\B_2(M);\mathbb{Z}/2)$ for $M$ a closed smooth manifold. But a simpler argument is available. Start by observing that if \eqref{cotainf} were to fail, then there would exist a continuous section $\sigma$ for $\varepsilon_2^{S^k}$. In such a situation we could consider the composite
$$S^k\stackrel{\alpha}{\longrightarrow}\CF_2(S^k)\stackrel{\widetilde{\sigma}}{\longrightarrow}e_2^{-1}(\CF_2(S^k))\hookrightarrow PS^k$$
where $\alpha(x)=(x,-x)$ and $\widetilde{\sigma}$ would be the ($\Z/2$-equivariant) pull-back of $\sigma$ under \eqref{lases}.  The adjoint of this composite would then yield a homotopy $H\colon S^k\times [0,1]\to S^k$ between the identity $H(-,0)$ and the antipodal map $H(-,1)$, and which would in addition satisfy the relation
\begin{equation}\label{biequiv}
H(x,t)=H(-x,1-t). 
\end{equation}
But this is impossible since the identity on $S^k$ (which has degree 1) cannot be homotopic to the presumed map $H(-,1/2)$ which, in view of \eqref{biequiv}, would factor as $$S^k\stackrel{\mathrm{proj}}\longrightarrow\mathbb{RP}^k\to S^k,$$ and would therefore have even degree.

\m
The above argument, as well as the closely related proof of \propref{p:further} below, were pointed out to the authors by Peter Landweber.

\begin{prop}\label{p:further}
Let $k$ be a positive odd integer. For $X=S^k$ and $i\geq2$, $\Sv(\varepsilon_i)\geq1$. Further, $\hspace{.4mm}\Sv(\varepsilon_i) = 1$ provided $k=1$.
\end{prop}

\begin{proof}[Proof of Proposition~$\ref{p:further}$]
The second assertion follows from the first one in view of Proposition~\ref{p:conn} and the first part of Remark~\ref{fupo}. To prove the first assertion, we derive a contradiction from the assumption that $\varepsilon_i$ admits a global continuous section $\sigma$. Consider the map $c\colon S^k\to(S^k)^{J_i}$ given as the composite
$$S^k\stackrel{\alpha}{\longrightarrow}\CF_i(S^k)\stackrel{\widetilde{\sigma}}{\longrightarrow}e_i^{-1}(\CF_i(S^k))\hookrightarrow (S^k)^{J_i}.$$ Here $\alpha(x)=(x,zx,z^2x,\ldots,z^{i-1}x)$, where $z\in S^1$ is a primitive $i\th$ root of unity acting on $S^k$ in the standard way (recall $k$ is odd), and $\widetilde{\sigma}$ is the $\Sigma_n$-equivariant section of the map $e_i:e^{-1}_i(\CF_i(S^k))\to\CF_i(S^k)$ obtained as the pull-back in \eqref{lases} of the assumed $\sigma$. Thus, for each $x\in S^k$, $c(x)$ is a multipath $\{c_j(x)\}_{j=0}^{i-1}\in (S^k)^{J_i}$, where each $c_j(x)$ is a path in $S^k$ starting at a point $s(x)\in S^k$ and ending at $z^jx$, for a continuous map $s: S^k \to S^k$. Note that the equivariance of $\widetilde{\sigma}$ gives
\begin{equation}\label{inicio}
c_j(zx) = c_{j+1}(x)
\end{equation}
for all $x\in S^k$---here the value of $j$ is to be interpreted modulo $i$. Then the map $H\colon S^k\times [0,1]\to S^k$ defined by $H(x,t)=c_0(x)(t)$ is a homotopy starting at $s$ and ending at the identity. In particular, $s\colon S^k\to S^k$ has degree 1. The contradiction comes by observing that the degree of $s$ would be divisible by $i$. Indeed, \eqref{inicio} gives
$$
s(zx)=c_0(zx)(0)=c_1(x)(0)=s(x),
$$
so that $s$ factors as $S^k\stackrel{\mathrm{proj}}{\longrightarrow}L^k(i)\to S^k$ where $L^k(i)$ is the standard lens space $S^k/(\mathbb{Z}/i)$.
\end{proof}

\begin{cor}\label{laevi}
The known equality $\TC^S_2(S^k)=2$ $($valid for any 
integer $k>0)$ extends to $\TC^S_n(S^k)=2(n-1)$ for $k=1$.
\end{cor}

\begin{remark}
The first conclusion in Proposition~\ref{p:further} is partially extended by Karasev-Landweber's result in~\cite{KL} asserting that $\Sv(\varepsilon_3^{S^k})\geq1$ for $k$ not of the form $4\cdot3^e$ with $e\geq0$. Note that the conjectural~(\ref{c:TCnSpherek}) would in fact sharpen the above estimate to an equality.
\end{remark}

\end{document}